\documentclass[preprint,12pt]{elsarticle}

\usepackage{amsthm,amsmath,amssymb}

\theoremstyle{plain}
\newtheorem{theorem}{Theorem}
\newtheorem{lemma}[theorem]{Lemma}

\theoremstyle{definition}

\theoremstyle{remark}

\journal{Discrete Mathematics}

\begin{document}

\begin{frontmatter}



\title{\bf Winning Strategies in Multimove Chess (i,j)}


\author{Emily Berger \qquad Alexander Dubbs}

\address{Department of Mathematics\\[-0.8ex]
Massachusetts Institute of Technology\\[-0.8ex] 
emilyritaberger@gmail.com \qquad alex.dubbs@gmail.com\\}

\begin{abstract}
 We propose a class of chess variants, Multimove Chess (i,j), in which White gets i moves per turn and Black gets j moves per turn. One side is said to win when it takes the opponent's king. All other rules of chess apply. We prove that if (i,j) is not (1,1) or (2,2), and if $i \geq \min(j,4)$, then White always has a winning strategy, and otherwise Black always has a winning strategy.
\end{abstract}


\end{frontmatter}

\maketitle

\section{Introduction}

Chess has driven substantial research in discrete mathematics. Numerous lay periodicals publish mate-in-$k$-moves puzzles, but much chess research is more sophisticated. The $n$-queens problem is a classic: place $n$ queens on an $n\times n$ chessboard such that none attack each other. This problem has been generalized to $d$-dimensional chessboards \cite{Barr2006}. In \cite{Grinstead1991} the authors find an upper bound on the minimum number of queens required to inhabit or attack every square on an $n\times n$ board. The knight's tour problem is also famous: how many ways can a knight visit every square on a chess board exactly once? The answer was found by \cite{Lobbing1996} and has been adapted to various modified chess boards, listed in and added to by \cite{DeMaio2007}. Rook polynomials count the number of ways that $k$ rooks can be placed on a subset of the chess board with none attacking each other, and they are related to the Laguerre orthogonal polynomials \cite{Gessel2013}. In \cite{Elkies2005}, the author discusses old and new enumerative chess problems, in which one counts the solutions to a chess problem that can be solved in more than one way.

Chess has also influenced much academic research in artificial intelligence.  Now computers can beat anyone in chess by brute force \cite{Max2011}, and can even play perfect checkers \cite{Max2011}. However, no perfect chess algorithm has ever been found, and finding one is a longstanding open problem.

Chess variants are as old as the game itself, and a number are listed in \cite{Wiki2014} and \cite{Var2014}. We propose algorithms for playing certain variants that are guaranteed to win regardless of their opponents' strategies.

We introduce a class of chess variants, ``Multimove chess (i,j).'' For a given (i,j) pair, Multimove Chess (i,j) is a game with exactly the same rules as standard chess except that White gets i moves per turn and Black gets j moves per turn. We say that the game is won when one side takes the other's king. With the exceptions of (1,1) and (2,2), we prove that Multimove Chess (i,j) can always be won by White or by Black (no stalemates occur). (1,1) is regular chess, and we conjecture that (2,2) is a first player game, and that this hypothesis can be proven by a computer. The following table summarizes our findings; ``White'' means that White can always win for a given (i,j), ``Black'' means that Black can always win for a given (i,j), and a ``?'' denotes an open problem.

\begin{center}
\begin{tabular}{|c|c|c|c|c|c|}
\hline
 Black/White & $j=1$ & $j = 2$ & $j = 3$ & $j = 4$ & $j > 4$\\
\hline
$i=1$ & ? & Black & Black & Black & Black\\
\hline
$i=2$ & White & ? & Black & Black & Black\\
\hline
$i=3$ & White & White & White & Black & Black\\
\hline
$i = 4$ & White & White & White & White & White\\
\hline
$i > 4$ & White & White & White & White & White\\
\hline
\end{tabular}
\end{center}
\vskip .1 in

\begin{theorem} If $(i,j) \neq (1,1)$ and $(i,j) \neq (2,2)$, and if $i \geq \min(j,4)$, Multimove Chess (i,j) can always be won by White. Otherwise it can always be won by Black. \end{theorem}
\begin{proof} Combine lemmas 2-10 from Section 2.\end{proof}

\section{Proofs}

We say that ``(i,j) is White'' or ``(i,j) is Black'' if White or Black has a strategy that is guaranteed to win at Multimove Chess (i,j). We recommend reading this section using a chessboard with labeled rows and columns.

\begin{lemma} (2,1) is White. \end{lemma}
\begin{proof} White moves the b1 knight to a3 and then to b5. Then Black makes {\it any move.} Then the White knight moves b5-c7 and then to e8 for checkmate.\end{proof}

\begin{lemma}(3,1) is White.\end{lemma}
\begin{proof} White moves the b1 knight to a3 and then to b5, and pushes the pawn at h2. Black again can make any move. The White knight moves b5-c7-e8 for checkmate.\end{proof}

\begin{lemma} (3,2) is White.\end{lemma}
\begin{proof} White moves the knight b1-c3, the pawn e2-e3, and the queen d1-f3. The White knight at c3 can take the Black king in three moves, by moving to d5, then to c7, then to e8. There is no way for Black to capture the c3 knight, so it must move the king. It can do so by advancing any of the d7, e7, or f7 pawns one or two spaces and then moving the king to d7, e7, or f7 (respectively). In any of these situations the White queen at f3 can capture the Black king in two moves.\end{proof}

\begin{lemma} (3,3) is White.\end{lemma}
\begin{proof}  We claim if White's first turn consists of the three moves: pawn e2-e3, queen d1-f3 and knight b1-c3, that White may take Black's king on the following turn. Note that in this position, Black cannot take White's king in three moves. The White knight at c3 can reach e8 in three moves (for example c3-d5-f6-e8), and therefore on Black's first turn it must capture the White knight or move its king so that it is not captured at e8.

Black can only capture the knight using at least three moves, either with (i) the b7 or d7 pawn (b7-b5-b4-c3 or d7-d5-d4-c3), (ii) the bishop at f8 (push g pawn and f8-g7-c3 or push e7 pawn and f8-b4-c3), (iii) the queen (push e7 pawn and d8-f6-c3 or push c7 pawn and d8-a5-c3), or (iv) the knight at g8 (for example g8-f6-e4-c3).  In all cases, the White queen has not been taken nor has the f7 pawn been moved, and hence the queen can move f3-f7-e8 to capture the king. Therefore, Black must move its king. 

Assuming Black does not take the White knight, there are only two positions that Black's king can reach in three moves that cannot be taken by the White knight at c3 on the following turn: d8 and f8. If the Black king moves to d8, one of the Black pawns in columns c, d, or e must be pushed, and the Black queen and then king moved. In any of these cases, the White queen may move f3-f6-e7-d8, taking the king. In the similar case of f8, either the Black e7 or g7 pawn must be pushed, and the Black bishop and then king moved. The White queen moves f3-f7-f8 to take the Black king.
\end{proof}

\begin{lemma} (i,j) where $i \geq 4$ and $j\geq 1$ is White.\end{lemma}
\begin{proof} The trick is that if White is given 4 moves before Black is given a single move, White can win. Move the White knight b1-a3-b5-c7-e8 to take the Black king.\end{proof}

\begin{lemma} (1,2) is Black. \end{lemma}
\begin{proof} White has 10 initial pieces it can move: 8 pawns and 2 knights. We discuss these cases separately; (A) pawns at a2, b2, c2, g2, h2 or knight at b1, (B) pawn at f2 or knight at g1, (C) pawn at d2, and (D) pawn at e2.
\begin{itemize}
\item[](A) {In any of these cases e5 is not being attacked by White after its first move. Black's first move is knight b8-c6-e5, now able to take the king on its next turn. White is incapable of moving the king, nor may it take the knight at e5, and hence Black must win.}
\item[](B) {In both cases, White is not attacking b4, and therefore Black moves its knight b8-c6-b4. Regardless of what White does, Black may take the king on the following turn with either b4-d3-e1 or b4-d3-f2.}
\item[](C) {Black pushes the c2 pawn and moves its queen d8-a5. White may not move its king to avoid capture and only has one move with which to block, which is insufficient as Black has two moves, taking the blocking piece at b4, c3, or d2, and then taking the king at e1.}
\item[](D) {In the case in which White pushes its e2 pawn 1 or 2 spaces, Black responds by moving its knight b8-c6-e5. On Black's next turn, Black may take the king at e1, and therefore White is forced to move its king e1-e2, since it may not take the Black knight. Black responds by pushing its pawn d7-d6 and moves its bishop c8-g4. Regardless of White's next move, Black may take the king on its following turn using the knight at e5 for spaces d3, e3, f3, e1 or bishop at g4 for e2. Note that neither the g4 bishop nor the e5 knight may be taken by White in one move, nor does blocking the g4 bishop matter since it is distance 2 from e2, moving g4-f3-e2.}
\end{itemize} \end{proof}

\begin{lemma} (1,3) is Black. \end{lemma}
\begin{proof} Regardless of White's first move, on Black's first turn it moves its knights b8-c6 and g8-f6, and its a7 pawn to a6. In three moves, any of the spaces in the box defined by: c1-3, d1-3, e1-3, f1-3, g1-3 can be taken by a Black knight. In two moves, White's king may not escape the above box, nor may it take either knight.\end{proof}

\begin{lemma} (2,3) is Black. \end{lemma}
\begin{proof} The strategy is to usually use the following standard opening for Black: pawn e7-e6, queen d8-f6, knight b8-c6. However, Black will use a specialized opening if White's first two move sequence is part of a four-move sequence that can take the Black king, given that Black uses the standard opening ((A2), (B2), (B3), (C)), or in situations in which the standard opening is inefficient ((A1), (B1)). Those two-move sequences for White are: (A1) knight from b1 to a4, (A2) knight from b1 to b5, c4, d5, or e4, (B) push the e2 pawn and move the White queen to f3 (B1), g4 (B2), or h5 (B3), and (C) push the c2 pawn one space and move the White queen to b3. We will demonstrate that Black can win in all of these situations:

\begin{itemize}
\item[] (A1) and (A2) Black takes the White knight with two moves of a pawn (if it only requires one, push the h7 pawn also), and then moves its knight b8-c6. In three moves, this knight can take the White king, regardless of the Black pawn structure, so as White cannot take the Black knight, it must move its king. The only three places to which the White king can move are d7, e7, and f7. In any of these cases, the Black knight can take the White king in three moves, without ever using a4, b5, c4, d5, or e4, where there may be a Black pawn.

\item[] (B1) Black pushes its pawn e7-e5, pushes its queen from d8-f6, and knight from b8-c6. Note that from this position, the Black knight at c6 may capture the White king at e1 in three moves, therefore White must take the knight or move its king to avoid a loss.

Assume White takes the Black knight. If White uses its queen to take both the Black queen and knight, Black may take White's king by bishop f8-b4-d2-e1. If White takes the knight at c6 with its queen or bishop, and either leaves the king where it is or moves its king to either d1 or e2, we can always take the White king with the Black queen, even if White can use a move to block the f column.

If White does not take the Black knight at c6, it must move its king to escape the Black knight. The White king can move to e2, e3 (if the e2 pawn is moved to e4), d1, d2 (if the d2 pawn is moved), d3, and f1 (if the bishop is moved). If the king takes two moves and goes to e3 or d3, it can be captured by the Black queen. If it moves to d2 or e2, it can be captured by the Black knight. It Black moves the bishop so that the king may move to f1, the Black queen takes the White king. So Black must move its king to d1. If White then moves its queen to f6, its king will be taken by the c8 bishop via pawn to d6 and bishop c8-g4-d1. If White moves its king to d1 and its queen to somewhere other than f6, Black moves the queen from f6 to f3 to d1, possibly stopping to take the White queen. Now we assume the king has been moved to d1 and the White queen is not moved. In this case, Black takes the White queen at f3 with its own queen at f6 and then takes the king, possibly stopping at e2 to take a piece.

\item[] (B2)  Black uses its knight to capture White's queen: g8-f6-g4 and moves its second knight b8-c6. Either knight can capture the king at e1 in three moves, and it is impossible to capture both knights in two moves as each knight requires two moves by White to be captured; therefore the king must move. In two moves, on its second turn, White's King can only reach d1-3, e1-3, f1-3; the combination of knights can reach all of those positions, and neither can be taken provided the king has used a move since they each require two moves from White to be captured, and therefore Black takes the king.

\item[] (B3) Black pushes the pawn g7-g6 to block the queen, and moves both knights b8-c6 and g8-f6. The knight at c6 can take the king at e1, and therefore the knight must be taken or the king moved. The knight at c6 may be taken by the queen h5-(b5,c5,d5)-c6 or by the bishop f1-b5-c6. In either case, Black takes the queen or bishop at c6 with the pawn at d7-c6, then the queen moves d8-d2-e1 to take the king. If the White king tries to move, it is impossible for it to escape the range of both Black knights, neither of which can be taken, so Black wins.

\item[] (C) On its first turn, Black moves a pawn d7-d5, moves the queen d8-d6 and moves a knight b8-c6. The knight can take the White king in three moves, so White must either take the knight or move its king. The only way White can take the Black knight is using two moves of its queen, allowing the Black queen to take the White king on the next turn. If instead White moves its king, the Black queen can take it at d2, e2, and f2 (if the relevant pawns are moved) and at c2. If the White king moves to d1, it can use its remaining move to try to defend itself from the Black queen by moving a pawn, but it is to no avail, the Black queen or the Black bishop at c8 will still take the White king (the bishop is used if the White e2 pawn is pushed).
\end{itemize}

Now we assume that Black uses the opening pawn e7-e6, queen d8-f6, knight b8-c6, and that during its first turn White has not made two moves as enumerated above in (A1-2), (B1-3), and (C). 

Since Black has progressed a knight from b8-c6, and every situation in which White makes two moves toward a four-mate move on Black's standard opening is handled above, on its second turn White must capture the Black knight or move its king. It must use its first turn to prepare, but it may not use the combinations above. If White captures the knight, it does so in three or four moves using (D) a knight, (E) a bishop, (F) a rook, (G) the queen, or (H) a pawn. If it uses only three moves, the fourth can be used to block before Black starts its second turn.

\begin{itemize}
\item[] (D) Using the knight at b1 has been covered by (A1-2) above. Using the White knight at g1 to capture Black's knight at c6 takes three moves, leaving White with one additional move. For any additional move White makes (with its pawns or c6 knight), the queen can move from f6 to e1 in three moves and capture the king as White cannot sufficiently block Black's queen, nor can its king move.

\item[] (E) In order for the White bishop to capture the Black knight at c6, either the e2 or g2 pawns must be pushed, and bishop moved to b5 or g2 respectively. White can then capture the knight with the bishop and has one remaining move. White may move the king to e2 or f1, push its f2 pawn, move its g1 knight to f3, or block with its c6 bishop or d1 queen. In all cases the Black queen can still capture the White king in three moves.

\item[] (F) To capture the knight with a rook requires four moves leaving the queen to effortlessly capture the king via f6-f2-e1. A pawn must be pushed two, say a2-a4, then rook a1-a3-c3-c6. The rook at h1 is equivalent.

\item[] (G) In order for the White queen to take the knight at c6, one of the c2, d2, or e2 White pawns must be pushed. In all of these cases, the queen can take the knight in three or four moves and may not take Black's queen. If it takes four moves, the black queen may move f6-f2-e1 and take the king, so we now only consider those cases where it takes three moves. They are queen d1-a4-c6 or d1-f3-c6. Going through f3 is case (B1) above, so we only must consider going through a4. In this case, White pushes its c2 pawn and may either block or move the king with its fourth move. If White moves the king to d1, Black may move its queen f6-f2-e1-d1 or f6-f3-e2-d1 to take the White king. If White blocks with its f2 pawn, g1 knight, or queen, the Black queen can always get to f2 in two moves and take the king on its third.

\item[] (H) White can capture the c6 knight with either the b2 or d2 pawn, say b2-b4-b5-c6. The queen can then capture the king via f6-f2-e1, possibly stopping to take a piece along the way. If the White king moves it is still easily taken by the Black queen.
\end{itemize}

We now consider the case in which the Black knight is not taken, and therefore the White king must move to avoid capture by the Black knight. The only positions that the king may reach in two turns (four total moves) that cannot be taken by the knight at c6 on its next turn (three moves) are d1, e4, f1, and g2; we consider these cases separately in (I), (J), (K), (L) respectively.

\begin{itemize}
\item[](I) For the White king to move to d1, the White queen must be moved. To free the queen, either the c2, d2, or e2 pawn must be moved (cases (i), (ii), (iii), respectively):
\begin{itemize}
\item[](i) Suppose the c2 pawn is pushed, the queen may move to c2, b3, or a4, and the king moves to d1. The case of the queen to b3 has been discussed above in (C), so we only consider a4 and c2. In both cases, consider the attack by Black f6-f3-e2-d1. White may only block this attack with its queen, or by pushing the f2 pawn to f4. In the case of a block by the queen, by a pawn, or by the g1 knight, Black may move f6-d4-d2-d1 or f6-f2-e1-d1 and capture the king.
\item[](ii) If the d2 pawn is pushed, the queen may move to d2 if the pawn is pushed to d3, or d2 or d3 if the pawn is pushed to d4, and the king moves to d1. White has one additional move, but Black can move f6-f2-e1-d1 or f6-d4-(d2 or d3)-d1 and capture the king, unless the White queen is at d2 with a White pawn at d3, and White pushes the f2 pawn. In that case, the Black queen moves f6-h4-e1-d1 to take the king.
\item[](iii) If the e2 pawn is pushed, the queen then moves to e2 (we considered the cases of f3, g4, and h5 in (B1-3) respectively above) and the king moves to d1. Consider the attack by Black's queen of f6-c3-c2-d1; White has one additional move to try and block this attack and must place a piece at either d4 or e5 to do so. The only pieces that can reach those positions are the d2 or e2 pawns (provided the initial e2 pawn was pushed to e4). In the case that the pawn is pushed d2-d4 to block, Black can capture via f6-d4-d1. In the case that the pawn e4-e5 is pushed, Black can capture via f6-e5-e2-d1.
\end{itemize}

\item[](J) If the White king moves to e4, it must have gotten there in four moves by pushing the d2, e2, or f2 pawn. In any case the Black queen takes the exposed White king.

\item[](K) In order for the White king to move to f1, either the e2 or g2 pawn must be pushed, the bishop on f1 moved, and then the king moved to f1. White has one additional move available, but none can sufficiently block nor take the Black queen. The Black queen can take two pieces in the f column before capturing the king at f1. Before White's final turn, there is only a pawn at f2, and with one additional move White can add at most one piece to the f column.

\item[](L) There is one way in which the White king can get to g2. It is if both the f2 and g2 pawns are pushed one or two spaces and the king moves e1-f2-g2. The Black queen can move f6-f4-f3-g2 to capture the king. 
\end{itemize}
\end{proof}

\begin{lemma} (i,j) where $i\leq 3$ and $j\geq 4$ is Black.\end{lemma}
\begin{proof} Note that White cannot take either of Black's knights nor Black's king in its first three moves (if Black has never moved). In at most four of Black's moves however, Black can put one of its knights on any of the squares in the box defined by: c1-3, d1-3, e1-3, f1-3, g1-3. So after its first three moves, White must make sure that its king is on none of those squares. This is impossible; in order for the king to move at all, a pawn must be pushed, giving the king only two moves - not enough to escape the box. As well, it would take three moves to prepare for castling (moving a pawn and then bishop, and knight). So the White king is taken. \end{proof}

\end{document}